\documentclass{amsart}

\usepackage{amssymb,amsbsy,amsthm,amsmath,graphicx,epsfig,enumerate,xcolor}

\newtheorem{thm}{Theorem}[section]
\newtheorem{lemma}[thm]{Lemma}

\newtheorem{cor}[thm]{Corollary}

\usepackage{version} 
\numberwithin{equation}{section}
\usepackage{enumerate}
\usepackage[bf]{caption}     
\usepackage{bm}  

\theoremstyle{definition}
\newtheorem{definition}[thm]{Definition}

\theoremstyle{remark}

\newcommand{\N}{\mathbb{N}}

\newcommand{\Z}{\mathbb{Z}}
\newcommand{\R}{\mathbb{R}}
\newcommand{\Q}{\mathbb{Q}}
\newcommand{\C}{\mathbb{C}}

\newcommand{\aveN}{\frac{1}{N}\sum_{n=1}^N}

\newcommand\veps\varepsilon

\allowdisplaybreaks

\newcommand{\ddd}{\delta}

\newcommand{\ds}{\displaystyle}
\newcommand{\eee}{\varepsilon}

\newcommand{\fff}{{\varphi}}

\newcommand{\mmm}{\mu}

\newcommand{\oo}{\infty}

\newcommand{\sm}{\setminus}
\newcommand{\sse}{\subset}

\newcommand{\tttt}{\tau}

\newcommand{\tUU}{\widetilde{U}}

\renewcommand{\lll}{\lambda}
\newenvironment{case}{\left\{ \begin{array}{cl} }{ \end{array} \right.}

\definecolor{skyblue}{rgb}{0,0.4,0.6}
\definecolor{red}{rgb}{0.6,0,0}
\definecolor{green}{rgb}{0,0.6,0}
\definecolor{aquam}{rgb}{0.5,1.0,1.0}
\definecolor{bbrown}{rgb}{0.75,0.38,0.15}
\definecolor{Cyan}{rgb}{0,0.6,0.6}
\definecolor{Darkblue}{rgb}{0,0,0.9}
\definecolor{Dodgerblue2}{rgb}{0,0.5,1}
\definecolor{Green}{rgb}{0,0.5,0.1}
\definecolor{dGreen}{rgb}{0,0.5,0}
\definecolor{Kahki}{rgb}{1,1,0.5}
\definecolor{Magenta}{rgb}{1,0,1}
\definecolor{bMagenta}{rgb}{1,.6,1}
\definecolor{Orange}{rgb}{0.8,0.3,0}
\definecolor{dOrchid}{rgb}{0.7,0.2,0.4}
\definecolor{Orchid}{rgb}{1,0.5,1}
\definecolor{Purple}{rgb}{0.65,0.07,0.85}
\definecolor{Royalblue}{rgb}{0.6,0.85,0.87}
\definecolor{Tan}{rgb}{0.54,0.42,0.23}
\definecolor{bTan}{rgb}{0.94,0.82,0.63}
\definecolor{Turquoise}{rgb}{0,0.85,0.87}
\definecolor{Yellow}{rgb}{1,1,0}
\definecolor{bYellow}{rgb}{1,1,0.6}
\definecolor{bRed}{rgb}{1,0.7,0.7}
\definecolor{dRed}{rgb}{0.7,0,0}
\definecolor{dRed}{rgb}{1,0,0}
\definecolor{boxcolb}{rgb}{0.87,0.77,0.75}
\definecolor{boxcol}{rgb}{0.6,0.85,0.87}
\definecolor{boxcolgreen}{rgb}{0.64,0.93,0.79}
\definecolor{boxcolaa}{rgb}{.75,.99,.70}
\definecolor{boxcolbb}{rgb}{0.39,0.50,0.56}
\definecolor{boxcolcc}{rgb}{1,0.81,0.65}
\definecolor{yy}{rgb}{0.43,0.21,.18}
\definecolor{gA}{gray}{0.5}
\definecolor{gB}{gray}{0.8}
\definecolor{gC}{gray}{0.9}

\begin{document}

\title[weighted square averages in $L^1$]{Divergence of weighted square averages in $L^1$}

\author{Zolt\'an Buczolich}
\address{Department of Analysis, ELTE E\"otv\"os Lor\'and University, P\'azm\'any P\'eter S\'et\'any 1/c, 1117 Budapest, Hungary\newline\indent
ORCID ID: 0000-0001-5481-8797 }
\email{ zoltan.buczolich@ttk.elte.hu}
\urladdr{ http://buczo.web.elte.hu }
\author{Tanja Eisner}
\address{Institute of Mathematics, University of Leipzig\newline
P.O. Box 100 920, 04009 Leipzig, Germany}
\email{eisner@math.uni-leipzig.de}

\thanks{The first listed
author  was supported by the Hungarian
National Foundation for Scientific Research Grant 124003.}


\maketitle

\begin{abstract}
We study convergence of ergodic averages along squares with polynomial weights. For a given polynomial $P\in\Z[\cdot]$, consider the set of all $\theta\in[0,1)$ such that for every 
 ergodic  system $(X,\mu, T)$ there is a function $f\in L^1(X,\mu)$ such that the weighted averages along squares
$$
\aveN e(P(n)\theta)T^{n^2}f
$$
diverge on a set with positive measure. We show that this set is residual and includes the rational numbers as well as a dense set of Liouville numbers. 


On one hand, this extends the divergence result for unweighted averages along squares in $L^1$ of  the first author and Mauldin; on the other hand, it shows  that the convergence result for linear weights for squares in $L^p$, $p>1$,  due to Bourgain  
as well as the second author and 
Krause does not hold for $p=1$.
\end{abstract}

\section{Introduction}

Originally motivated by physics, ergodic theory became an independent mathematical discipline in the 1930s with the appearence of the classical ergodic theorems due to von Neumann and Birkhoff. Since then, 
ergodic theorems have been ex\-tended and generalized in many directions and surprising connections to other areas of mathematics have been discovered. We 
mention  three  
major directions: the multiple, the subsequential and the weighted 
ergodic theorems.

\begin{itemize}
\item
Multiple ergodic theorems, dealing with averages of the form
$$
\aveN T^n f_1\cdot T^{2n}f_2 \cdots T^{kn}f_k
$$
and introduced by Furstenberg in his celebrated ergodic theoretic proof of Szemer\'edi's theorem, have deep connections to, e.g., combinatorics, harmo\-nic analysis, number theory, group theory and form an active area of research, see Furstenberg \cite{F,F-book}, Furstenberg, Katznelson \cite{FK},  Bourgain \cite{B-double},
Host, Kra \cite{HostKra}, Ziegler \cite{Z}, Leibman \cite{L}, Bergelson, Leibman, Lesigne \cite{BLL}, Tao \cite{T}, Walsch \cite{Wa}, Donoso,  Sun \cite{DS}.

\item
Subsequential ergodic theorems concern averages 
$$
\aveN T^{k_n}f
$$
for a subsequence $(k_n)$ of $\N$ and are natural from the physical point of view.  They have been studied by Furstenberg, Bourgain, Wierdl and others using methods from harmonic analysis and number theory, see Bourgain \cite{B}, Wierdl \cite{W}, Nair \cite{N}, 
Bellow, Losert \cite{BeL},  
 Rosenblatt, Wierdl \cite{RW}, Akcoglu, Bellow, Jones, Losert, Reinhold-Larsson, Wierdl \cite{Akcoglu-etal}, 
Krause \cite{K},  Zorin-Kranich \cite{ZK-primes}, Eisner \cite{E}.

\item
Weighted ergodic theorems concerning averages  of the form
$$
\aveN a_n T^nf 
$$
with $(a_n)\subset \C$ go back to the Wiener-Wintner theorem \cite{WW} and are connected to both additive number theory and the recent Sarnak conjecture, see Green, Tao \cite{GT12}, Sarnak \cite{sarnak-lectures}, El Abdalaoui, Kulaga-Przymus, Lema\'nczyk, de la Rue \cite{AKLR}.

The Wiener-Wintner theorem 
deals with pointwise convergence 
for the family of linear weights $(e(n\theta))$, $\theta\in\R$, where 
$e(x):=e^{2\pi i x}$, see Assani  \cite{AWW}.  
Lesigne \cite{Lesigne,lesigneWW} extended it to polynomial weights of the form $(e(P(n)))_{n=1}^\infty$ for real polynomials $P$, see also Frantzikinakis \cite{Fr} and Assani \cite{AWW} and the generalization to nilsequences by
Host, Kra \cite{HK} and  Eisner, Zorin-Kranich \cite{EZK}.  For  topological versions of the Wiener-Wintner theorem and its generalizations see  Robinson \cite{Robi}, Assani \cite[Chapter 2.6]{AWW} and Fan  \cite{Fantww}.

For more weighted ergodic theorems see Bellow, Losert \cite{BeL},  
Berend, Lin, Rosenblatt, Tempelman \cite{BLRT}, Bourgain, Furstenberg, Katznelson, Ornstein \cite{B-RTT}, Lin, Olsen, Tempelman \cite{LOT}, Eisner, Lin \cite{EL},  Fan \cite{FanwB}. 
For  a different type of weighted ergodic theorems with 
arithmetic weights we refer to Cuny, Weber \cite{[CW]} and Buczolich \cite{[Bu]}.

 A related   question to weighted ergodic theorems is about the convergence of ergodic series 
$\sum_{n=0}^\oo a_n f (T^nx).$ We can refer to Cohen and Lin
\cite{CoLi}, Fan \cite{Fanaec} and Cuny and Fan \cite{CuFa}. 
In Izumi \cite{Izu} the question
about the convergence of
$\sum \frac{1}{n}f (T^n x)$ was first raised. Positive answers
would be improvements of the ergodic theorems. But there are no such answers in
general. Negative answers were given by Halmos \cite{Halm}, Dowker and Erd\H os
\cite{DErd}, moreover by Kakutani and Petersen \cite{KaPe}.

%
%
\end{itemize}

Surprisingly, pointwise convergence of the simplest mixture of the subsequential and weighted ergodic averages, namely polynomial averages with polynomial weights, is still open in general. 
Pointwise convergence of the unweighted square averages $\aveN T^{n^2}f$ was proved by Bourgain \cite{B86,B88,B} for $L^p$, $p>1$, answering a question of Bellow and Furstenberg, see also Krause \cite{K}, whereas divergence in $L^1$ was shown by Buczolich, Mauldin \cite{BuM}, extended by LaVictoire \cite{LaV} to all monomials. The averages along 
squares
 with linear weights  
\begin{equation}\label{eq:ave-square-lin}
\aveN e(n\theta)T^{n^2}f
\end{equation}
were treated by Bourgain \cite{B88-2,B} and Eisner, Krause \cite{EK} where pointwise convergence in $L^p$, $p>1$, was shown for every $\theta$, partially uniformly in $\theta$. Note that 
a Wiener-Wintner type result for the weighted double averages 
$$
\aveN e(n\theta)\, T^{n}f_1 \cdot T^{2n}f_2
$$
for bounded functions  
was obtained by Assani, Duncan, Moore \cite{ADM} from Bourgain's unweighted double recurrence result \cite{B-double}. See also Assani, Moore \cite{AM} for an extension to polynomial weights.

In this paper we study convergence of the averages (\ref{eq:ave-square-lin}) for $f\in L^1$  and show that for many $\theta$ the weight $(e(\theta n))$ is $L^1$-universally bad, extending the mentioned above convergence results in $L^p$, $p>1$, by Bourgain and Eisner, Krause as well as the divergence result of unweighted averages by Buczolich, Mauldin in $L^1$. More precisely, we treat polynomial weights $(e(P(n)\theta))$ with $P\in\Z[\cdot]$.

\medskip


\begin{definition}\label{def:weighted-ave}
Let $(a_n)\subset \C$ and $(k_n)$ be a subsequence of $\N$. We say that the pair $((a_n),(k_n))$ is \emph{$L^1$-universally bad} if for every   
ergodic invertible   system $(X,\mu,T)$  on a nonatomic standard probability space  there is $f\in L^1(X,\mu)$ such that the weighted averages along $(k_n)$
\begin{equation}\label{eq:weighted-ave}
\aveN a_nT^{k_n}f
\end{equation}
diverge on a set of positive measure.
\end{definition}

Some authors use aperiodic instead of ergodic in the definition of universally bad sequences, see for example \cite{AWW},  \cite{Be} and \cite{BeL}.  One can show analogously to the unweighted case that the two versions of the definition are equivalent.  
In this paper we use the ergodic version and this way we follow for example  \cite{BuM}, \cite{RW} and \cite{LaV}.

Thus,  using 
the terminology of Definition \ref{def:weighted-ave},   Buczolich and Mauldin \cite{BuM} showed that $((1), n^2)$ is $L^1$-universally bad.  
 This result was slightly generalized in  \cite{LaV} by LaVictoire to show that $((1),n^d)$
is $L^1$-universally bad and in \cite{Butor} by Buczolich to show that 
  for any polynomial $q(n)$ of degree two with integer coefficients the sequence
  $((1), q(n))$ is $L^1$-universally bad. Both these generalizations were based on the 
original argument of    \cite{BuM}. It is still an open question whether for a
"general" polynomial $q(n)$ of degree at least three with integer coefficients the
sequence  $((1), q(n))$ is $L^1$-universally bad. 

In this paper we restrict our attention to the base case $k_{n}=n^{2}$.
For the above mentioned slightly generalized cases which depend on the argument of
 \cite{BuM}  similar results to Theorem \ref{thm:main} can be obtained. 
The case of general polynomials  $q(n)$ of degree at least three with integer coefficients
is a challenging unsolved problem.

\medskip 

Our main result is the following.

\begin{thm}\label{thm:main}
Let $P\in\Z[\cdot]$ be a polynomial and let $\mathcal{M}$ be the set of all $\theta\in[0,1)$ such that the pair $(e(\theta P(n)),(n^2))$ is $L^1$-universally bad. Then the following assertions hold.
\begin{enumerate}
\item[(a)] $\Q\cap [0,1)\subset \mathcal{M}$.
\item[(b)] $\mathcal{M}\setminus \Q$  contains a dense set of Liouville numbers.
\item[(c)] $\mathcal{M}$ is a dense $G_\delta$ subset of $[0,1]$ and therefore   residual. 
\end{enumerate}
\end{thm}

\section{Main tool}

Next we need to recall some definitions and theorems, and tweak some arguments from  \cite{BuM}.
We start with the definition of $M-0.99$ distributed random variables.
\begin{definition}\label{conmkilenc}
For a positive integer $M$
we say that a function
or a random variable,
$X:[0,1)\to \R$ is {\it $M$$-$$0.99$ distributed}
  if
$X(x)\in  \{ 0,0.99,0.99\cdot \frac{1}2,...,
0.99\cdot 2^{-M+1}  \}$,
 and
$\lll ( \{ x\in[0,1):X(x)=0.99\cdot 2^{-l}  \})
=0.99 \cdot 2^{-M+l-1}$ for $l=0,...,M-1$,  where $\lambda$ denotes the Lebesgue measure.
\end{definition}

 An easy calculation shows that
\begin{equation}\label{*mean}
u=\int_{[0,1]} X(x)d\lambda (x)=0.99^{2}\cdot M \cdot 2^{-M-1}>
0.9\cdot M \cdot 2^{-M-1}.
\end{equation}

On the probability space
$([0,1),\lambda )$ we can
consider  pairwise independent  $M-0.99$-distributed
random variables $X_{h}$ for $h=1,...,K$ for a sufficiently
large $K$. Assume that $u$ denotes the mean of these variables.

By the weak law of large numbers
$\lambda \left  \{ x:\left  |
\frac{1}K\sum_{h=1}^{K} X_{h}(x)-u
\right|\geq \frac{u}2\right  \}\to 0$ as $K\to\oo$.
Given $\ddd>0$
 we can select $K$ so large that
 \begin{equation}\label{*meanlow}
 \lambda \left  \{ x:\frac{1}K
\sum_{h=1}^{K}
X_{h}(x)\geq \frac{u}2
\right  \}>1-\ddd.
 \end{equation}

With slight change of notation we recall   Theorem 8 from \cite{BuM}.

 \begin{thm}
\label{*c1c}
Given $\delta >0$, $M$ and $K$ there exist
$\tau_0 \in  {\ensuremath {\mathbb N}},$ 
 which defines a translation on $[0,1)$ by
$T(x)=x+\frac{1}{\tau_0}$ modulo $1$, 
$  E_{\delta}  {\subset} [0,1)$
with
$\lambda (  E_{\delta}  )<\delta $, 
a measurable function
$g:[0,1)\to [0,+ {\infty})$ with $\int_{[0,1)}gd\lambda <K\cdot 2^{-M+2},$
and pairwise independent  $M$$-$$0.99$-distributed
random variables
$ X_{h},$ $h=1,...,K$ 
 defined on $([0,1),\lambda)$  such that 
for all
$x\in [0,1) {\setminus}   E_{\delta} $ there exists
$N_{x}$ satisfying
\begin{equation}\label{*eqNx}
\frac{1}{N_{x}}\sum_{k=1}^{N_{x}}
g(T^{k^{2}}(x))>\sum_{h=1}^{K}  {X}_{h} (x).
\end{equation}
\end{thm}

This theorem is highly non-trival and its proof in \cite{BuM} is quite technical.
In  \cite{BM} there is some heuristic outline of this argument. One of the key elements
of this argument is coming from number theory and is related to the randomness of
quadratic residues  (see \cite{[Kur]} and \cite{[Kur2]}). 
Another version of the argument of \cite{BuM}  can be found in
\cite{LaV}.

Our main tool will be the following corollary of Theorem \ref{*c1c}. Here and later, we denote by $\Z_d$ the cyclic group $\{0,\ldots, d-1\}$.  On $\Z_{d}$ 
we denote by $T$ the shift transformation $Tx=x+1\mod d$.

\begin{cor}\label{cor:black-box}
For every ${\mathbf N}_1\in \N$ and every $\veps, C>0$ there exist an arbitrarily large $\tau\in \N$, a set $E\subset \Z_{\tau}$ with proportion less than $\veps$ in $\Z_{\tau}$, a positive bounded function
$f$ on $\Z_{\tau}$ with $\int f\,d\mu< \veps$  ($\mu$ being the normalised counting measure), and ${\mathbf N}_2>{\mathbf N}_1$ such that the inequality
\begin{equation}\label{*bbC}
\max_{N\in[{\mathbf N}_1, {\mathbf N}_2]} \aveN (T^{n^2}f )(l)>C
\end{equation}
holds for every $l\in \Z_{\tau}\setminus E$.
\end{cor}

\begin{proof}[Verification of Corollary \ref{cor:black-box} based on methods of  \cite{BuM}]

We can now modify slightly the proof of Theorem 1 on p. 1528 of \cite{BuM}.
We select $p$ such that 
\begin{equation}\label{*eqpsel}
\frac{1}{p}<\frac{\eee}{4}\text{ and }{32\cdot C}\cdot {4^{-p}}<\frac{\eee}2.
\end{equation}

We let $M_{p}=4^{p}$,  $\delta=\frac1p$  and select $K$ such that for $M_p-0.99$-distributed
random variables $X_{h}$,  $h=1,...,K$ we have \eqref{*meanlow} satisfied and hence
for this $K$ by \eqref{*mean}  we have
\begin{equation}\label{*Uest}
\lambda (U_p')>1-\ddd\text{ for }
U_p':=\left  \{ x:
\frac{1}K\sum_{h=1}^{K} X_{h}(x)>\frac{0.9}2
\cdot M\cdot 2^{-M-1}\right  \}.
\end{equation}

By Theorem \ref{*c1c} used with $\delta =\frac{1}p,$
$M_{p}$ and $K$ there exist $\tau_0 \in \N,$
$E_{1/p}\sse [0,1)$ and a periodic transformation
$T:[0,1)\to [0,1),$ $T(x)=x+\frac{1}{\tau_0 }$ modulo $1,$
$g:[0,1)\to [0,+\oo)$,
$X_{h}$ pairwise independent
$M_{p}-0.99$-distributed random variables defined on
$[0,1)$ such that $\lambda (E_{1/p})<\frac{1}p$
and for all $x\in [0,1)\sm E_{1/p}$
there exists $N_{x}$ such that
$$\frac{1}{N_{x}}\sum_{k=1}^{N_{x}}
g(T^{k^{2}}x)>\sum_{h=1}^{K}X_{h}(x)$$
and
$\int_{[0,1)}fd\lambda <K\cdot 2^{-M_{p}+2}.$

One can observe that on p. 1527 of \cite{BuM} at the beginning of the proof of Theorem 8 one can choose $A={\mathbf N}_1$ instead of $A=1$
and this implies that $N_{x}\geq {\mathbf N}_1$ holds in \eqref{*eqNx}.
By using a slightly larger exceptional set ${\widetilde{E}_{1/p}}\supset { {E_{1/p} }}$
still satisfying $\lll(\widetilde{E}_{1/p})<1/p$ we can select ${\mathbf N}_2$
such that ${\mathbf N}_1\leq N_{x}\leq {\mathbf N}_2$
holds for any $x\not\in \widetilde{E}_{1/p}$.
One can also observe that any integer multiple of $\tau_{0}$ could also be used, so $\tau_{0}$
can be arbitrarily large.

Put $U_{p}=U_{p}'\sm
\widetilde{E}_{1/p}.$
Then $\lambda (U_{p})>1-\frac{2}p$ and for
$x\in U_{p}$ there exists $ N_{x}\in [ {\mathbf N}_1, {\mathbf N}_2]$ such that
$$\frac{1}{N_{x}}
\sum_{k=1}^{N_{x}}g(T^{k^{2}}x)>\sum_{h=1}^{K}
X_{h}(x)>K\cdot \frac{0.9}2\cdot M_{p}\cdot 2^{-M_p-1}.$$
Thus letting $t_{p}=K\cdot \frac{0.9}2\cdot M_{p}\cdot 2^{-M_p-1}$,
and 
\begin{equation}\label{*Udef}
\tUU_{p}=\left  \{ x:\sup_{N}
\frac{1}N
\sum_{k=1}^{N}g(T^{k^{2}}x)>t_{p}\right  \}
\end{equation}
we have $U_{p}\sse \tUU_{p}$ and hence
$\lambda (\tUU_{p})>1-\frac{2}p.$
On the other hand,
$$ \frac{\int g d\lambda }{t_{p}}
<
\frac{K\cdot 2^{-M_{p}+2}}
{K\cdot \frac{0.9}2\cdot M_{p}\cdot 2^{-M_p-1}}
< \frac{32}{M_{p}}.$$

Now set $ G =\frac{C}{t_{p}}\cdot g.$ Then \eqref{*eqpsel} and the above inequality
imply that $\int  G   d\lll<\eee/2.$
We also put $ \overline{E} =[0,1)\sm {\widetilde {U}}_{p}$. Then $\lll( \overline{E} )<\eee/2$.

Next we need a sort of a transference argument to move the above results onto the integers.

Given $\fff:[0,1)\to \R$, $x\in [0,1/\tttt)$ put
$$\fff_{\tttt}(x)=\frac{1}{\tttt}\sum_{j=0}^{\tttt-1}\fff\Big(x+\frac{j}{\tttt}\Big) \text{ and }\lll_{\tttt}=\tttt\cdot \lll|_{[0,1/\tttt)},$$
where $\lll|_{[0,1/\tttt)}$ is the restriction of the Lebesgue measure onto $[0,1/\tttt)$.

By this notation we have 
\begin{equation}\label{*2*a}
\int_{[0,1)}\fff(x)d\lll(x)=\int_{[0,1/\tttt)}\sum_{j=0}^{\tttt-1}\fff\Big(x+\frac{j}{\tttt}\Big)d\lll(x)
\end{equation}
$$=\int_{[0,1/\tttt)}\fff_{\tttt}(x)\cdot \tttt d\lll(x)=\int_{[0,1/\tttt)}\fff_{\tttt}(x)d\lll_{\tttt}(x).$$

For $x\in [0,1/\tttt)$ and a measurable set $A\sse [0,1)$ let
$\mmm_{\tttt,x}(A):=\chi_{A,\tttt}(x)=\frac{1}{\tttt}\sum_{j=0}^{\tttt-1}\chi_{A}\Big(x+\frac{j}{\tttt}\Big)$.
Then
using \eqref{*2*a} with $\fff(x)=\chi_{ \overline{E} }(x)$ 
we obtain that $$\int_{[0,1/\tttt)}\mmm_{\tttt,x}( \overline{E} )d\lll_{\tttt}(x)
 = 
\int_{[0,1)}\chi_{ \overline{E} }(x)d\lll(x)=\lll( \overline{E} )<\frac{\eee}{2}.$$
This implies that
\begin{equation}\label{*wml1}
\lll_{\tttt}\Big (\Big \{x\in [0,1/\tttt):\mmm_{\tttt,x}( \overline{E} )\geq \eee \Big  \}\Big )<\frac{1}{2}.
\end{equation}

Using \eqref{*2*a} with $\fff(x)= G (x)\geq 0$
we obtain similarly
 $$\int_{[0,1/\tttt)} G _{\tttt} (x)d\lll_{\tttt}(x)=\int_{[0,1)} G (x)d\lll(x)<\frac{\eee}{2}.$$
and this implies that
\begin{equation}\label{*wml2}
\lll_{\tttt}\Big (\Big \{x\in [0,1/\tttt): G _{\tttt}(x)\geq \eee \Big  \}\Big )<\frac{1}{2}.
\end{equation}
 Since $\lll_{\tttt}([0,1/\tttt))=1$ by \eqref{*wml1} and \eqref{*wml2} we can select
 an $x^{*}\in [0,1/\tttt)$  such that $\mmm_{\tttt,x^{*}}( \overline{E} )<\eee$
 and $ G _{\tttt}(x^{*})<\eee.$

 Now we can define $f:\Z\to[0,+\oo)$, periodic by $\tttt$  such that $f(l)= G (x^{*}+\{ l/\tttt \})$ where $\{ . \}$ denotes fractional part, this also defines a function on $\Z_{\tttt}$ which, for ease of notation is also denoted by $f$.
 
 To define the exceptional set $E$ we say that $l\in\Z$ belongs to $E$
 iff $x^{*}+\{ l/\tttt \}\in  \overline{E} $.
 Then $f$ and $E$ are both periodic by $\tttt$, $\int f d\mmm<\eee$,
 $\mmm(E)<\eee$, the definition of $ G $ and
 \eqref{*Udef} imply \eqref{*bbC}.
\end{proof}

\section{Weighted Conze Principle}

\begin{definition}
Let $(X,\mu)$ be a probability space and let $(T_N)$ be a sequence of bounded linear operators on  $L^1(X,\mu)$. Define the corresponding \emph{maximal operator} $T^*$ by 
$$
(T^*f)(x):=\sup_{n\in\N} |(T_nf)(x)|\in[0,\infty],\quad x\in X,\ f\in L^1(X,\mu).
$$
We say that $(T_N)$ satisfies a \emph{weak $(1,1)$ maximal inequality} if there exists a constant $C>0$ such that for every $f\in L^1(X,\mu)$ and every $\lambda>0$ 
\begin{equation}\label{eq:max-ineq}
\mu(T^*f>\lambda)\leq\frac{C\|f\|_1}\lambda.
\end{equation}
\end{definition}

The following is a corollary of Sawyer's variation of Stein's principle, see \cite[Corollary 1.1]{sawyer}.
\begin{lemma}[Sawyer]\label{lemma:sawyer}
Let $(X,\mu, \mathbf{T})$ be an ergodic measure-preserving dynamical system and let $(T_N)$ be a sequence of bounded linear operators on  $L^1(X,\mu)$ commuting with the Koopman operator $\mathbf{T}$. Assume that $(T_N)$ does not satisfy a weak $(1,1)$ maximal inequality. Then there exists a function $f\in L^1(X,\mu)$ such that $T^*f=\infty$ a.e., and, in particular, $(T_Nf)$ diverges a.e.
 (Moreover, the set of such functions is residual in the Baire category sense.)
\end{lemma}

We will need the following variation of Conze's principle.

\begin{thm}[Weighted Conze principle]\label{thm:weighted-Conze}
Let $(a_n)\subset\C$ be bounded 
and $(k_n)$ be a subsequence of $\N$. Let $C\leq\infty$ be minimal such that for every 
system $(X,\mu,T)$  and  every $f\in L^1(X,\mu)$ 
\begin{equation}\label{eq:weighted-max-ineq}
\mu\left(\sup_{N\in\N}\left|\aveN a_nT^{k_n}f\right|>\lambda\right)\leq\frac{C}{\lambda}\|f\|_1\quad \forall \lambda>0
\end{equation}
holds. Then $C<\infty$  if and only if there exists an  ergodic invertible   system $(X,\mu,T)$  on a nonatomic standard probability space   
such that for every $f\in L^1(X,\mu)$, the weighted averages (\ref{eq:weighted-ave}) converge a.e.  Equivalently, $C
=\infty$ if and only if $((a_n),(k_n))$ is $L^1$-universally bad.
\end{thm}
\noindent 
%
The proof is an adaptation of the argument in Rosenblatt, Wierdl \cite[Proof of Theorem 5.9]{RW}  which is based on the original work by Conze \cite{[Cz]}.
\begin{proof}
The ``only if'' part is trivial. To show the ``if'' part,  assume that for an   ergodic invertible   system $(X,\mu,T)$  on a nonatomic  standard probability space   and every $f\in L^1(X,\mu)$, the weighted averages (\ref{eq:weighted-ave}) converge a.e. By Lemma \ref{lemma:sawyer}, there is $C>0$ such that (\ref{eq:weighted-max-ineq}) holds for every $f\in L^1(X,\mu)$.
 
Since all nonatomic  standard probability spaces are isomorphic,
it suffices to show that for every 
(invertible) transformation $\tilde{T}$ on $(X,\mu)$, (\ref{eq:weighted-max-ineq}) holds for the same constant $C$ and every $f\in L^1(X,\mu)$. 
Take such $\tilde{T}$. By the Halmos conjugacy lemma, there exists a sequence $(S_l)$ of invertible transformations such that $\lim_{l\to\infty}S_lTS_l^{-1}=\tilde{T}$ in the weak topology. 
Then by a standard approximation argument,  the Koopman operators on $L^1(X,\mu)$ (which we denote by the same letter) satisfy 
$\lim_{l\to\infty}S_lTS_l^{-1}=\tilde{T}$ in the  strong operator topology. Thus $\lim_{l\to\infty}S_lT^nS_l^{-1}=\tilde{T}^n$ in the strong operator topology for every $n\in\N$.

Let now $f\in L^1(X,\mu)$, $\lambda>0$ and $M\in\N$. By monotonicity it suffices to show that  
$$
\mu\left(\sup_{1\leq N\leq M}\left|\aveN a_n\tilde{T}^{k_n}f\right|>\lambda\right)\leq\frac{C}{\lambda}\|f\|_1.
$$
Since $(a_n)$ is bounded, $\aveN a_nS_lT^{k_n}S_l^{-1}f$ converges to $\aveN a_n\tilde{T}^{k_n}f$ in $L^1(X,\mu)$ for every $N\in\N$, and the same of course holds for  $\sup_{1\leq N\leq M}$ of the absolute value. 
Since for every sequence $(g_n)\subset L^1(X,\mu)$ converging in norm to $g\in L^1(X,\mu)$ one has
$\lim_{n\to\infty}\mu(x: g_n(x)>\lambda)=\mu(x: g(x)>\lambda)$, it suffices to show that  
$$
\mu\left(\sup_{1\leq N\leq M}\left|\aveN a_nS_lT^{k_n}S_l^{-1}f\right|>\lambda\right)\leq\frac{C}{\lambda}\|f\|_1
$$
or equivalently, by the measure-preserving property of $S_l$,
$$
\mu\left(\sup_{1\leq N\leq M}\left|\aveN a_nT^{k_n}g\right|>\lambda\right)\leq\frac{C}{\lambda}\|g\|_1
$$
for $g:=S_l^{-1}f$. But this holds by the definition of the constant $C$ and monotonicity.
\end{proof}

\section{Proof of  the $G_\delta$ property}

We first prove that $\mathcal{M}$ is a $G_\delta$ set. The denseness follows from  Theorem \ref{thm:main} a) or b) and will be proven in the following sections.  

\begin{proof}[Proof of Theorem  \ref{thm:main} c) assuming Theorem  \ref{thm:main} a) or b)]

Observe that by Theorem \ref{thm:weighted-Conze}, $\theta\in\mathcal{M}$ if and only if for every $C\in\Q$ there  exist  a system $(X,\mu,T)$, a function $f\in L^1(X,\mu)$ with $\|f\|_1=1$ and $\lambda>0$ such that 
\begin{equation}\label{eq:categ}
\mu\left(\sup_{N\in\N}\left|\aveN e(P(n)\theta)T^{n^2}f\right|>\lambda\right)>\frac{C}\lambda.
\end{equation}

Since $\ds \aveN e(P(n)\theta)T^{n^2}f(x)$ appears at many places in this proof we 
introduce the notation $$\sigma_{N}^{(\theta)}f(x):=\aveN e(P(n)\theta)T^{n^2}f(x).$$

By monotonicity of the sets 
$
\{x\in X: \max_{N\leq k}
| \sigma_{N}^{(\theta)}f(x)
|>\lambda
\},
$ 
(\ref{eq:categ}) is equivalent to the existence of $k\in \N$ such that 
$$
\mu\left(\max_{N\leq k}\left| \sigma_{N}^{(\theta)}f(x)\right|>\lambda\right)>\frac{C}\lambda.
$$
Thus we obtain
$$
\mathcal{M}=\bigcap_{C\in\Q} \bigcup_{(X,\mu,T),f,\lambda,k}   \Big \{\theta\in[0,1):\, \mu \Big (\max_{N\leq k}\Big | \sigma_{N}^{(\theta)}f\Big |>\lambda\Big ) > \frac{C}\lambda \|f\|_1\Big \},
$$
where under the union sign $(X,\mu,T)$ denotes an arbitrary measure-preserving system, $f\in L^1(X,\mu)$ an arbitrary function with $\|f\|_1=1$, $\lambda>0$ an arbitrary real number and $k$ an arbitrary natural number.
It remains to show that each of the sets on the right
is open, and for that it suffices to show that for given $(X,\mu,T)$, $f$, $\lambda$ and $k$, the $1$-periodic function
$$
g:\R\to [0,1],\quad 
g(\theta):= \mu \left(\max_{N\leq k}\left| \sigma_{N}^{(\theta)}f\right|>\lambda\right)
$$
is continuous.

Let $\theta\in \R$, $(\theta_j)_{j=1}^\infty\subset\R$ with $\lim_{j\to\infty}\theta_j=\theta$ and let $\veps>0$.  Using the elementary estimate $|e(x)-e(y)|=|e^{2\pi i x}-e^{2\pi i y}|\leq 2 \pi |x-y| $,  by  
$$
\left| \sigma_{N}^{(\theta)}f-
 \sigma_{N}^{(\theta_j)}f\right|\leq  2 \pi  \sup_{ n\in  [1,N]}|P(n)||\theta-\theta_j|T^{n^2}|f|,
$$
we obtain for every $j\in\N$
\begin{eqnarray*}
\mu \left(\max_{N\leq k}\left| \sigma_{N}^{(\theta)}f\right|>\lambda+\veps\right)
&\leq&
\mu \left(\max_{N\leq k}\left| \sigma_{N}^{(\theta_j)}f\right|>\lambda\right)\\
&+&
\mu( 2\pi \sup_{ n\in [1,k]}|P(n)||\theta-\theta_j|T^{n^2}|f|>\veps).
\end{eqnarray*}
Since $T$ is $\mu$-preserving and $\|f\|_1=1$, the last summand on the right hand side equals 
$$
\mu\left(
\sup_{ n\in [1,k]} 2\pi |P (n) ||\theta-\theta_j||f|>\veps
\right)\leq \frac{\sup_{ n\in [1,k]} 2\pi |P (n) ||\theta-\theta_j|}{\veps}.
$$
Therefore we have for every $\veps>0$
$$ 
\mu \left(\max_{N\leq k}\left| \sigma_{N}^{(\theta)}f\right|>\lambda+\veps\right)
\leq
\liminf_{j\to\infty}
\mu \left(\max_{N\leq k}\left| \sigma_{N}^{(\theta_j)}f\right|>\lambda\right),$$
implying, by letting $\veps\to 0$,
\begin{eqnarray*}
g(\theta)
&\leq&
\liminf_{j\to\infty} g(\theta_j).
\end{eqnarray*}
Analogously one shows $g(\theta)\geq\limsup_{j\to\infty} g(\theta_j) $, implying the continuity of $g$ and completing the proof.

\end{proof}

\section{Reduction}

We first reduce Theorem \ref{thm:main}{ a) and b)}  to the following. 

\begin{thm}\label{thm:main-claim}
Let $P\in\Z[\cdot]$ be a polynomial with $P(0)=0$.
For every rational number $\frac{p}{q}\in[0,1)$ and every $k\in\N$ there exist $r>0$, a system $(X,\mu,T)$ and a positive function $f\in L^\infty(X,\mu)$ satisfying 
\begin{equation}\label{eq:to-show}
\int f
\leq 1, 
\quad 
\mu\left(\sup_N \left|\aveN e(\theta P(n))T^{n^2}f\right|>k\right)>1-\frac1k
\end{equation}
for every $\theta \in [\frac{p}{q}-r,\frac{p}{q}+r]$. 
\end{thm}

 The system $(X,\mu,T)$ will be a shift on $\Z_{\tau q^2}$ 
with suitable $\tau$ and $q$. 
We now show that Theorem \ref{thm:main-claim} implies Theorem \ref{thm:main}. 

\begin{proof}[Proof of Theorem \ref{thm:main} {a) and b)} based on Theorem \ref{thm:main-claim}]
Assume that Theorem \ref{thm:main-claim} holds and let $P\in\Z[\cdot]$. Since multiplication by a non-zero constant does not affect divergence, we can assume without loss of generality that $P(0)=0$. 

 Let $\frac{p_1}{q_1}\in[0,1)$ be arbitrary. 
The claim for $k:=1$ implies the existence of an arbitrarily small $r_1>0$ such that there  exist  a system $(X_1,\mu_1,T_1)$ and a positive $f_1\in L^\infty (X_1,\mu_1)$ such that 
$$
\int f_1 \leq 1,\quad \mu_1\left(\sup_N \left|\aveN e(\theta P(n))T_{1}^{n^2}f_{1}\right|>1\right)>0
$$
holds for every $\theta \in[\frac{p_1}{q_1}-r_1,\frac{p_1}{q_1}+r_1]$. Take now an arbitrary rational number $\frac{p_2}{q_2}\in  (\frac{p_1}{q_1}-r_{1},\frac{p_1}{q_1}+r_{1})$. The claim for $k:=2$ implies the existence of an arbitrarily small $r_2$, a  system $(X_2,\mu_2,T_2)$ and a positive function $f_2\in L^\infty (X_2,\mu_2)$ such that 
$$
\int f_2 \leq 1,\quad \mu_2\left(\sup_N \left|\aveN e(\theta P(n))T_{2}^{n^2}f_{2}\right|>2\right)>\frac12
$$
holds for every $\theta \in [\frac{p_2}{q_2}-r_2,\frac{p_2}{q_2}+r_2]
\sse  [\frac{p_1}{q_1}-r_1,\frac{p_1}{q_1}+r_1] $.
 Repeating the procedure we get a  nested  sequence of rapidly decreasing intervals $[\frac{p_k}{q_k}-r_k,\frac{p_k}{q_k}+r_k]$ such that 
$$
 \{  \theta  \}  :=\bigcap_{k=1}^\infty\left[\frac{p_k}{q_k}-r_k,\frac{p_k}{q_k}+r_k\right]
$$
satisfies the following property: For every $k\in\N$ there exist a  system $(X,\mu,T)$ and a positive function $f_{ k }\in L^\infty (X,\mu)$ with property (\ref{eq:to-show})
 satisfied with $f_{k}$ instead of $f$.   
  Since $1-\frac{1}{k}\geq \frac{1}{2}$ for $k\geq 2$ 
by the weighted Conze principle (Theorem \ref{thm:weighted-Conze}), $((e(\theta P(n))),(n^2))$ is $L^1$-univerally bad. Note that we have some freedom in the above construction of $\theta$ by choosing each $\frac{p_k}{q_k}$ and by taking $r_k$ as small as we wish.

To show (a), by taking in the above construction $p_k:=p_1$ and $q_k:=q_1$ for every $k\in\N$, we have $\theta:=\frac{p_1}{q_1}$ which was arbitrary, and (a) follows.

To show (b), take in the above construction $\frac{p_k}{q_k}$ being all different
and decrease $r_k$ 
such that $r_k<\frac{1}{q_k^k}$ and $r_k<|\frac{p_k}{q_k}-\frac{p_{k-1}}{q_{k-1}}|$ hold. Then $\theta\notin\{\frac{p_1}{q_1},\frac{p_2}{q_2},\ldots\}$ and  $\theta$ is Liouville, therefore irrational.

Thus Theorem \ref{thm:main} follows.
\end{proof}

\section{Proof of Theorem \ref{thm:main-claim}} 


\begin{proof}[Proof of Theorem \ref{thm:main-claim}]
Let $P\in \Z[\cdot]$ satisfy $P(0)=0$, let $\frac{p}{q}\in[0,1)$ be arbitrary and $k\in\N$.
We will define bounded positive functions $f_{1},\ldots,f_{q}$ and then $f:=\sum_{l=1}^{q} f_{l}$ on $\Z_{\tau q^2}$  for some large $\tau\in\N$ with the normalised counting measure and the right shift.
Denote $N_0:=1$.

\medskip

\underline{Construction of $  {N}_1$, $\tau_1$ and $f_{1}$ on $\Z_{\tau_{1}q^2}$.}

Let  $0<\veps_1\leq 1$ and $C_1>0$  to be chosen later.
By Corollary \ref{cor:black-box},  used with  ${\mathbf N}_1:=1$, 
there exist $\tau_{1}\in \N$ with $\tau_{1}\gg q$, the right shift transformation $T$ modulo $\tau_{1}$ on $\Z_{\tau_{1}}$ with the normalised counting measure, a set
$\tilde{E}_{1}\subset \Z_{\tau_{1}}$ with 
proportion less than $\veps_1$ in $\Z_{\tau_{1}}$,
a function
$\tilde{f}_{1}$ on $\Z_{\tau_{1}}$ with $\tilde{f}_{1}\geq 0$ and  $\int_{\Z_{\tau_{1}}} \tilde{f}_{1}\leq \veps_1$, and $N_1:={\mathbf N}_2>1={\mathbf N}_1$ such that the inequality
\begin{equation}\label{eq:black-box}
\max_{N\in[1, N_1]} \aveN T^{n^2}\tilde{f} _{1}(l)>C_1
\end{equation}
holds for every $l\in \Z_{\tau_{1}}\setminus \tilde{E}_{1}$.  

Consider $\Z_{\tau_{1}q^2}$ together with the normalised counting measure and the right translation which we denote by $T$ again.
Consider further the  function $f_{1}$ on $\Z_{\tau_{1}q^2}$ given by 
$$
f_{1}(l):=
\begin{cases}
\tilde{f}_{1}(l/q^2) & \quad \text{if } q^2\mid l,\\
0&\quad \text{otherwise}.
\end{cases}
$$

The support of this function is contained in 
$$
U_{1}:=\{0,q^2,2q^2,\ldots, (\tau_{1}-1)q^2
\}\subset \Z_{\tau_{1}q^2}
$$ 
and $\int_{\Z_{\tau_{1}q^2}}f_{1}\leq \frac{\veps_1}{q^2}
$ holds.
Observe
\begin{equation}\label{eq:zero}
(T^{n^2}f_{1})(l)=f_{1}\left(l+n^2\right) =0 \quad\text{whenever}\ \ q\nmid n \text{ and }l\in U_{1}.
\end{equation}
Moreover, for every $n\leq qN_1$  and every $\theta\in [\frac{p}{q}-r,\frac{p}{q}+r]$ observe
$$
|e(\theta P(n))-e(p/qP(n))|\leq2\pi |\theta P(n)-p/qP(n)|\leq 2\pi r \sup_{[1,qN_1]}|P| .
$$
By choosing $r<\frac1{4\pi \sup_{[1,qN_1]}|P|}$ we have by $P(0)=0$
\begin{equation}\label{eq:one-half}
|e(\theta P(n))-1|<1/2,\quad  \forall\theta\in \left[\frac{p}{q}-r,\frac{p}{q}+r\right],  
\quad\forall n\leq qN_1 \text{ divisible by } q.
\end{equation}

Define $E_1:=\{mq^2:\, m\in \tilde{E_1}\}$ and
consider now $l=mq^2\in U_1\setminus E_1$.
The inequalities (\ref{eq:black-box}), (\ref{eq:zero}) and (\ref{eq:one-half}) imply 
\begin{eqnarray*}
\nonumber\sup_{1\leq N\leq qN_1}\Re \aveN e(\theta P(n))  (T^{n^2}f_{1}) (l)
&=&
\nonumber\sup_{1\leq N\leq qN_1} \frac1N \sum_{n\leq N,\ q\mid n} \Re (e(\theta P(n)))  (T^{n^2}f_{1}) (l)\\
&\geq& 
\nonumber\frac12\cdot \sup_{1\leq N\leq q N_1} \frac1N \sum_{n\leq N,\ q\mid n}  (T^{n^2}f_{1}) (l)\\
&=&
\nonumber\frac1{2q}\cdot \sup_{1\leq N\leq qN_1} \frac{q}N \sum_{n\leq N,\ q\mid n}  \tilde{f}_{1} \left(m+\frac{n^2}{q^2}\right)\\
&\geq&
\nonumber\frac1{2q}\cdot \sup_{1\leq N\leq N_1} \aveN  (T^{n^2}\tilde{f}_{1}) (m)
> 
\frac{C_1}{2q}.
\label{eq:f_1,1}
\end{eqnarray*}

\medskip

\underline{Construction of $N_2$ and $f_{2}$ on $\Z_{\tau_{2}q^2}$.}

Let $\veps_2\leq \veps_1$ and $C_2 \geq C_1$ to be chosen later.
By Corollary \ref{cor:black-box},
there exist $\tau_{2}\in \N$ with $\tau_{2}\gg q$, the right shift transformation $T$ modulo $\tau_{2}$ on $\Z_{\tau_{2}}$ with the normalised counting measure, a set
$\tilde{E}_{2}\subset \Z_{\tau_{2}}$ with 
proportion less than $\veps_2$ in $\Z_{\tau_{2}}$, 
a function
$\tilde{f}_{2}$ on $\Z_{\tau_{2}}$ with $\tilde{f}_{2}\geq 0$ and  $\int_{\Z_{\tau_{2}}} \tilde{f}_{2}\leq \veps_2$, and $N_2={\mathbf N}_2>N_1={\mathbf N}_1$ such that the inequality
\begin{equation}\label{eq:black-box-N_2}
\max_{N\in[N_1,N_2]} \aveN T^{n^2}\tilde{f} _{2}(l)>C_2
\end{equation}
holds for every $l\in \Z_{\tau_{2}}\setminus \tilde{E}_{2}$. 


Stretch  $\tilde{f}_{2}$  to $\Z_{q^2\tau_{2}}$ as follows. Define  the  function $f_{2}$ on $\Z_{\tau_{2}q^2}$ given by 
$$
f_{2}(l):=
\begin{cases}
\tilde{f}_{2}\left(\frac{l-1}{q^2}\right) & \quad \text{if } q^2\mid (l-1),\\
0&\quad \text{otherwise}.
\end{cases}
$$
The support of this function is contained in 
$$
U_{2}:=\{1,q^2+1,2q^2+1,\ldots, (\tau_{2}-1)q^2+1
\}\subset \Z_{\tau_{2}q^2}
$$ 
and $\int_{\Z_{\tau_{2}q^2}}f_{2}\leq \frac{\veps_2}{q^2}$ holds.

Observe
\begin{equation}\label{eq:zero-f_1,2}
(T^{n^2}f_{2})(l)=f_{2}\left(l+n^2\right) =0 \quad\text{whenever}\ \ q\nmid n \text{ and }l\in U_{2}.
\end{equation}
 Moreover, for every $n\leq qN_2$  and every $\theta\in [\frac{p}{q}-r,\frac{p}{q}+r]$ observe
$$
\Big |e(\theta P(n))-e\left( \frac{p}{q}  P(n)\right)\Big |\leq2\pi \Big |\theta P(n)- \frac{p}{q}  P(n)\Big |\leq 2\pi r \sup_{[1,qN_2]}|P|.
$$
By choosing $r<\frac1{4\pi \sup_{[1,qN_2]}|P|}$ we have  by $P(0)=0$
\begin{equation}\label{eq:one-half-f_1,2}
|e(\theta P(n))-1|<1/2\quad  \forall\theta\in \left[\frac{p}{q}-r,\frac{p}{q}+r\right]
\quad\forall n\leq qN_2 \text{ divisible by } q.
\end{equation}

Define now $E_2:=\{mq^2+1:\, m\in \tilde{E_2}\}$ and
consider  $l = mq^2+1\in U_2\setminus E_2$.
 The inequalities (\ref{eq:black-box-N_2}), (\ref{eq:zero-f_1,2}) and (\ref{eq:one-half-f_1,2}) imply 
\begin{eqnarray*}
\nonumber\sup_{qN_1\leq N\leq qN_2}&\Re& \aveN e(\theta P(n))  (T^{n^2}f_{2}) (l)
\\
&=&
\nonumber\sup_{qN_1\leq N\leq qN_2} \frac1N \sum_{n\leq N,\ q\mid n} \Re (e(\theta P(n)))  (T^{n^2}f_{2}) (l)\\
&\geq& 
\nonumber\frac12\cdot \sup_{qN_1\leq N\leq q N_2} \frac1N \sum_{n\leq N,\ q\mid n}  (T^{n^2}f_{2}) (l)\\
&=&
\nonumber\frac1{2q}\cdot \sup_{qN_1\leq N\leq qN_2} \frac{q}N \sum_{n\leq N,\ q\mid n}  \tilde{f}_{2} \left(m+\frac{n^2}{q^2}\right)\\
&\geq&
\nonumber\frac1{2q}\cdot \sup_{N_1\leq N\leq N_2} \aveN  (T^{n^2}\tilde{f}_{2}) (m)
>  
\frac{C_2}{2q}.
\end{eqnarray*}

Now, choosing $C_2>c\|f_1\|_\infty$ with $c>0$ to be chosen later, we see that 
$$
\left\|\sup_{qN_1\leq N\leq qN_2} \Re  \aveN e(\theta P(n))  T^{n^2}f_{1}\right\|_\infty\leq \|f_1\|_\infty<\frac{C_2}{c}.
$$

\medskip


\underline{Construction of $\tau$ and $f$}

In such a fashion we construct for every $j\in\{1,\ldots,q^2\}$ and for $\veps_j\leq \veps_{j-1}, C_j>0$ to be chosen later (where $\veps_0:=1$)
an integer  
$\tau_{j}\gg q$, a set $E_{j}\subset \Z_{\tau_{j}q^2}$ with proportion less than $\veps_j$ in $Z_{\tau_jq^2}$, a natural number $N_j> N_{j-1}$  and positive bounded functions $f_{j}$ on $\Z_{\tau_{j}q^2}$  with $\int f_j\leq \veps_j$ 
 and supported on 
$$
U_{j}:=\{j-1,q^2+j-1,2q^2+j-1,\ldots, (\tau_{j}-1)q^2+j-1
\}\subset \Z_{\tau_{j}q^2}
$$
such that for every $l\in U_{j}\setminus E_j$ and every $\theta\in [\frac{p}{q}-r,\frac{p}{q}+r]$ with $r<\frac{1}{4\pi \sup_{[1,qN_j]}|P|}$
\begin{eqnarray}\label{eq:f_1,j}
\sup_{qN_{j-1}\leq N\leq qN_j}\Re \aveN e(\theta P(n))  (T^{n^2}f_{j}) (l)
>\frac{C_{j}}{2q}.
\end{eqnarray}
Moreover, we choose $C_j$ large  enough to satisfy
\begin{eqnarray}\label{eq:f_1,j-supnorm}
 \max_{k\leq j-1}\left\|\sup_{qN_{j-1}\leq N\leq qN_j}  \Re  \aveN  e(\theta P(n))  T^{n^2}f_{k}\right\|_\infty\\ \nonumber \leq 
\max\{\|f_1\|_\infty,\ldots, \|f_{j-1}\|_\infty\}<\frac{C_{j}}{c}.
\end{eqnarray}
\smallskip 

We now consider $\tau:=\tau_{1}\cdot\ldots\cdot\tau_{q^2}$ and extend the functions $f_{j}$ and the sets $E_{j}$ periodically to $\Z_{\tau q^2}$. (We use the same notation for these extensions.) These sets and functions have the unchanged  proportion in $\Z_{\tau q^2}$ and unchanged integrals, respectively. Moreover, (\ref{eq:f_1,j}) holds for every $l\in \Z_{\tau q^2}\setminus E_{j}$ and (\ref{eq:f_1,j-supnorm}) is still true.
We denote by $\mmm$ the normalised counting measure on $\Z_{\tau q^2}$.

Define now $f:=f_{1}+\ldots+f_{q^2}$. We have by the monotonicity of $\veps_j$
$$
\int_{\Z_{\tau q^2}}f\leq q^2\veps_1\leq 1
$$
by choosing $\veps_1\leq \frac1{q^2}$.
Define further $E:=E_{1}\cup\ldots\cup E_{q^2}$. Note that the proportion of $E$ in $\Z_{\tau q^2}$  is  less than $\sum_{j=1}^{q^2}\veps_j\leq q^2\veps_1 \leq \frac1{2k}$ by choosing $\veps_1\leq \frac1{2kq^2}$.



Take $\theta\in[\frac{p}{q}-r,\frac{p}{q}+r]$ with $r<\frac{1}{4\pi \sup_{[1,qN_{q^2}]}|P|}$ and $l\in \Z_{\tau q^2}\setminus E$. Then $l\in U_j$ for some $j\in\{1,\ldots, q^2\}$. Let $N$ satisfy $qN_{j-1}\leq N\leq qN_j$ and decompose 
 \begin{eqnarray}
\aveN e(\theta P(n))  (T^{n^2}f) (l)&=&
\aveN e(\theta P(n))  (T^{n^2}f_j) (l)\label{*dec}
\\ \nonumber
&+&\sum_{1\leq m<j} \aveN e(\theta P(n))  (T^{n^2}f_m) (l)\\ \nonumber
&+&
\sum_{j<m\leq q^2} \aveN e(\theta P(n))  (T^{n^2}f_m) (l)\\ \nonumber
&=& I_N(l)+II_N(l)+III_N(l).
\end{eqnarray}

Observe that by (\ref{eq:f_1,j}), since $l\in U_j\setminus E_j$, 
$$
\sup_{qN_{j-1}\leq N\leq qN_j}|I_N(l)|>\frac{C_j}{2q}.
$$
Moreover, by (\ref{eq:f_1,j-supnorm}), 
\begin{equation}\label{*trie}
\sup_{N\in \N} |II_N(l)|\leq \sum_{1\leq m<j} \|f_m\|_\infty\leq q^2 \frac{C_j}{c}\leq 
 \frac{C_j}{8q}
\end{equation}
if we choose $c\geq 8q^3$.

Finally, by the triangle inequality and $\mu(f_m>1)\leq \int f_m\leq \veps_m$ for each $m$, 
\begin{eqnarray}
\nonumber\mu\left(
\sup_{qN_{j-1}\leq N\leq qN_j}
|III_N|>q^2\right)
&\leq& 
\mu\left(\sum_{j<m\leq q^2\ } \sup_{qN_{j-1}\leq N\leq qN_j}
\aveN  T^{n^2}f_m>q^2\right)\\ \label{*trik}
&\leq& 
\sum_{j<m\leq q^2\ } \mu\left( \sup_{qN_{j-1}\leq N\leq qN_j}
\aveN   T^{n^2}f_m>1\right)\\ \nonumber
&\leq& 
\sum_{j<m\leq q^2\ }
\sum_{n\leq qN_j} \mu(T^{n^2}f_m>1)\\ \nonumber
&=& 
\sum_{j<m\leq q^2\ }\sum_{n\leq qN_j} \mu(f_m>1)\\ \nonumber
&\leq&
qN_j \sum_{j<m\leq q^2\ }\veps_m \leq q^3 N_j \veps_{j+1}\leq \frac1{2kq^2}
\end{eqnarray}
by choosing $\veps_{j+1}\leq  {1}/{(2kq^5N_j)}$. 

Denote now 
$
F_j:=\{l\in U_j:\, \sup_{qN_{j-1}\leq N\leq qN_j}|III_N|>q^2\}
$
and $F:=\cup_{j=1}^{q^2}F_j$.
By the above, $\mu(F)\leq\sum_{j=1}^{q^2}\frac1{2kq^2}=\frac1{2k}$. Thus if we assume in addition that $l\notin F$, the triangle inequality \eqref{*dec}, \eqref{*trie} and \eqref{*trik}  lead to 
\begin{eqnarray*}
\sup_{qN_{j-1}\leq N\leq qN_j}\aveN e(\theta P(n))  (T^{n^2}f) (l)
&>&
\frac{C_j}{2q} - \frac{C_j}{8q}-q^2\\
&=&\frac{C_j}{4q} -q^2\geq k
\end{eqnarray*}
if we choose $C_j\geq 4kq+q^2$.
Therefore, for every $l\in \Z_{\tau q^2}\setminus (E\cup F)$
$$
\sup_{N\in\N}\aveN e(\theta P(n))  (T^{n^2}f) (l)>k.
$$
Since $\mu(E\cup F)\leq 2\frac1{2k}=\frac1k$, the proof is complete.
 \end{proof}

\bigskip

\section{Further questions}

There are many open questions related to our results. We just mention two  here.
\begin{itemize}
\item[1)] Is every Liouville number universally $L^1$-bad? 
\item[2)] Is there an $L^1$-good number?
\end{itemize}

\section{Acknowledgement}

We thank the referee of this paper for making valuable comments and pointing out some useful references.

\parindent0pt

\end{document}